%% file: manuscript.tex
\documentclass[11pt]{article}

\usepackage{amsmath, amsfonts, amssymb, amsthm,  graphicx, mathtools, enumerate}

\usepackage[blocks, affil-it]{authblk}

\usepackage[numbers, square]{natbib}
\usepackage[CJKbookmarks=true,
            bookmarksnumbered=true,
			bookmarksopen=true,
			colorlinks=true,
			citecolor=red,
			linkcolor=blue,
			anchorcolor=red,
			urlcolor=blue]{hyperref}
\usepackage[usenames]{color}

\usepackage[letterpaper, left=1.2truein, right=1.2truein, top = 1.2truein, bottom = 1.2truein]{geometry}
\usepackage[ruled, vlined, lined, commentsnumbered]{algorithm2e}

\usepackage{prettyref,soul}

\newtheorem{lemma}{Lemma}[section]

\newtheorem{thm}{Theorem}[section]

\newtheorem{corollary}{Corollary}[section]
\newtheorem{remark}{Remark}[section]

\newrefformat{eq}{(\ref{#1})}
\newrefformat{chap}{Chapter~\ref{#1}}
\newrefformat{sec}{Section~\ref{#1}}
\newrefformat{algo}{Algorithm~\ref{#1}}
\newrefformat{fig}{Fig.~\ref{#1}}
\newrefformat{tab}{Table~\ref{#1}}
\newrefformat{rmk}{Remark~\ref{#1}}
\newrefformat{clm}{Claim~\ref{#1}}
\newrefformat{def}{Definition~\ref{#1}}
\newrefformat{cor}{Corollary~\ref{#1}}
\newrefformat{lmm}{Lemma~\ref{#1}}
\newrefformat{lemma}{Lemma~\ref{#1}}
\newrefformat{prop}{Proposition~\ref{#1}}
\newrefformat{app}{Appendix~\ref{#1}}
\newrefformat{ex}{Example~\ref{#1}}
\newrefformat{exer}{Exercise~\ref{#1}}
\newrefformat{soln}{Solution~\ref{#1}}
\newrefformat{cond}{Condition~\ref{#1}}

\def\text#1{\mbox{\rm #1}}

\newcommand{\norm}[1]{\|{#1} \|}

\newcommand{\fnorm}[1]{\|#1\|_{\rm F}}

\newcommand{\rank}{\mathop{\sf rank}}
\newcommand{\Tr}{\mathop{\sf Tr}}

\newcommand{\supp}{{\rm supp}}

\newcommand{\TV}{{\sf TV}}

\title{A General Decision Theory for Huber's $\epsilon$-Contamination Model
}
\author[1]{Mengjie Chen}
\author[2]{Chao Gao}
\author[3]{Zhao Ren}
\affil[1]{
University of North Carolina, Chapel Hill, mengjie@email.unc.edu
}
\affil[2]{
University of Chicago, chaogao@galton.uchicago.edu
}
\affil[3]{
University of Pittsburgh, zren@pitt.edu
}
\date{}

\begin{document}
\maketitle

\begin{abstract}
Today's data pose unprecedented challenges to statisticians. It may be incomplete, corrupted or exposed to some unknown source of contamination. We need new methods and theories to grapple with these challenges. Robust estimation is one of the revived fields with potential to accommodate such complexity and glean useful information from modern datasets. Following our recent work on high dimensional robust covariance matrix estimation, we establish a general decision theory for robust statistics under Huber's $\epsilon$-contamination model. We propose a solution using Scheff{\'e} estimate to a robust two-point testing problem that leads to the construction of robust estimators adaptive to the proportion of contamination. Applying the general theory, we construct robust estimators for nonparametric density estimation, sparse linear regression and low-rank trace regression. We show that these new estimators achieve the minimax rate with optimal dependence on the contamination proportion. This testing procedure, Scheff{\'e} estimate, also enjoys an optimal rate in the exponent of the testing error, which may be of independent interest.
\smallskip

\end{abstract}


\input{intro_re}

\input{test_re}

\input{upper_re}

\input{application_re}

\input{sup-norm_re}

\input{proof_re}

\section*{Acknowledgements}
The authors deeply thank Larry Wasserman for valuable comments and constructive suggestions to improve the
quality of the paper, especially for pointing out some early works related to our results by Yannis Yatracos​ as well as Luc Devroye and G{\'a}bor Lugosi.

\bibliographystyle{plainnat}
\bibliography{Robust}


\end{document}

%% file: intro_re.tex

\section{Introduction}

In Huber's pathbreaking papers \citep{huber1964robust,huber1965robust} on robust estimation theory, he proposed the $\epsilon$-contamination model
\begin{equation}
(1-\epsilon)P_{\theta}+\epsilon Q. \label{eq:central}
\end{equation}
Under this model, data are drawn from (\ref{eq:central}) with probability of $\epsilon$ to be contaminated by some arbitrary distribution $Q$. Given i.i.d. observations from (\ref{eq:central}), the objective is to estimate $\theta$ robust to the contamination from $Q$. It has been discussed in \cite{chen2015robust} that Huber's $\epsilon$-contamination model provides a favored framework which allows a joint study of statistical efficiency and robustness. In other words, the optimality of an estimator under Huber's $\epsilon$-contamination model indicates that it achieves statistical efficiency and robustness simultaneously. However, not much attention has been paid to this framework in nonparametric and high-dimensional statistics. Inspired by Tukey's work on data depth, we proposed a new concept, matrix depth, for robust estimation of covariance matrix in high dimension in our previous work \cite{chen2015robust}. We established the optimality of the proposed estimator under Huber's $\epsilon$-contamination model for several covariance matrix classes. This work leaves an important problem open: whether there exists a general rule for minimax rate under Huber's $\epsilon$-contamination model?

To address this problem in this paper, we investigate the following quantity
\begin{equation}
\inf_{\hat{\theta}}\sup_{\theta\in\Theta,Q}\mathbb{E}_{(\epsilon,\theta,Q)}L(\hat{\theta},\theta),\label{eq:minimax}
\end{equation}
the robust minimax risk for a given parameter space $\Theta$ and a loss function $L(\cdot,\cdot)$. The expectation $\mathbb{E}_{(\epsilon,\theta,Q)}$ is determined by the probability (\ref{eq:central}), and the supreme is taken over all $\theta\in\Theta$ and $Q$ in the class of all probability distributions. When the loss function takes the form of squared total variation distance, we can construct a general robust estimator $\hat{\theta}$, such that the robust minimax risk (\ref{eq:minimax}) is upper bounded by some universal constant times
\begin{equation}
\min_{\delta>0}\left\{\frac{\log\mathcal{M}(\delta,\Theta,\TV(\cdot,\cdot))}{n}+\delta^2\right\}\vee \epsilon^2,\label{eq:rate}
\end{equation}
where $\mathcal{M}(\delta,\Theta,\TV(\cdot,\cdot))$ denotes the $\delta$-covering number of $\Theta$ using the total variation distance. This rate (\ref{eq:rate}) consists of two parts. The first part is a common bias variance trade-off term in the classical decision theory without taking account of contamination. The second part is a term contributed by unknown contamination of the data. Comparing the rate (\ref{eq:rate})  to the general lower bound for the $\epsilon$-contamination model derived in our previous work \cite{chen2015robust}, we immediately find that (\ref{eq:rate}) is the minimax rate for the risk in (\ref{eq:minimax}). This is the main contribution of our paper.

The construction of rate-optimal robust estimators is enabled in this paper by a novel analysis of the robust testing procedure called Scheff{\'e} estimate that was first proposed in \cite{devroye2012combinatorial}. For the robust two-point testing problem, we propose a solution using Scheff{\'e} estimate, the testing error of which has a desired exponent leading to a rate-optimal estimation procedure. Our new testing theory has advantages over some classical ones. Under the contamination model, the classical Neyman-Pearson approach lacks robust property. The statistical performance of the likelihood ratio test can be compromised even when one contaminated point is included in the data. The robust testing theory established by Le Cam \citep{lecam1973convergence} and Birg\'{e} \citep{birge1979th} is based on Hellinger distance, which gives a sub-optimal rate for Huber's $\epsilon$-contamination model. The existing optimal testing function for the robust two-point testing problem was constructed by Huber himself \citep{huber1965robust}. However, his procedure depends on the knowledge of the contamination proportion $\epsilon$ in (\ref{eq:central}). As shown in our previous work, it is impossible to estimate $\epsilon$ when $Q$ is not specified. In comparison, our proposed testing function overcomes this and does not depend on $\epsilon$. This feature, together with its robustness and rate-optimal error exponent, makes our method superior to the previous ones.

The rest of the paper is organized as follows. We first introduce the robust testing problem in Section \ref{sec:test} and propose a solution using Scheff{\'e} estimate with a sharp testing error bound. In Section \ref{sec:upper}, we use this robust testing procedure to construct a general estimator that achieves the optimal rate for (\ref{eq:rate}). Then in Section \ref{sec:application}, we construct robust estimators for density estimation, sparse linear regression and low-rank trace regression as applications of the general theory. We show that for all these problems, our estimators achieve minimax optimal rates. Finally, we investigate a scenario when the loss function is not equivalent to total variation distance in the discussion section, Section \ref{sec:sup}. We show that the minimax rate for non-intrinsic loss functions may depend on $\epsilon$ in different ways. All technical proofs are gathered in Section \ref{sec:proof}.

We close this section by introducing the notation used in the paper. For $a,b\in\mathbb{R}$, let $a\vee b=\max(a,b)$ and $a\wedge b=\min(a,b)$. For an integer $m$, $[m]$ denotes the set $\{1,2,...,m\}$. Given a set $S$, $|S|$ denotes its cardinality, and $\mathbb{I}_S$ is the associated indicator function. For two positive sequences $\{a_n\}$ and $\{b_n\}$, the relation $a_n\lesssim b_n$ means that $a_n\leq Cb_n$ for some constant $C>0$, and $a_n\asymp b_n$ if both $a_n\lesssim b_n$ and $b_n\lesssim a_n$ hold. For a vector $v\in\mathbb{R}^p$, $\norm{v}$ denotes the $\ell_2$ norm and $\supp(v)=\{j\in[p]:v_j\neq 0\}$ is its support. For a matrix $A\in\mathbb{R}^{p_1\times p_2}$, $\rank(A)$ denotes its rank, $\text{vec}(A)$ is its vectorization and $\fnorm{A}=\norm{\text{vec}(A)}$ is the matrix Frobenius norm. When $A$ is an squared matrix, $\Tr(A)$ denotes its trace. For two probability distributions $P_1$ and $P_2$, their total variation distance is $\TV(P_1,P_2)=\sup_B|P_1(B)-P_2(B)|$, and their Hellinger distance is $H(P_1,P_2)=\left[\int\left(\sqrt{dP_1}-\sqrt{dP_2}\right)^2\right]^{1/2}$.

%% file: test_re.tex

\section{Robust Testing}\label{sec:test}

Given i.i.d. observations $X_1,...,X_n\sim P$, we consider the following robust two-point testing problem originally set up by Huber in \citep{huber1965robust}:
\begin{eqnarray*}
H_0: && P\in\left\{(1-\epsilon)P_0+\epsilon Q: Q\right\}, \\
H_1: && P\in\left\{(1-\epsilon)P_1+\epsilon Q: Q\right\}.
\end{eqnarray*}
In particular, $P_0$ and $P_1$ are two fixed distributions and $Q$ is in the class of all probability distributions. When $\epsilon=0$, it reduces to the classical two-point testing problem studied by Neyman and Pearson \citep{neyman1933problem}. They showed that the likelihood ratio test $\mathbb{I}\left\{\prod_{i=1}^n\frac{dP_1}{dP_0}(X_i)>t\right\}$ achieves the optimal testing error, which laid the foundation for modern hypothesis testing. However, the likelihood ratio test is not robust to cases when $\epsilon>0$. For example, when $P_0=N(\theta_0,I_p)$ and $P_1=N(\theta_1,I_p)$,  Neyman-Pearson testing statistic involves the calculation of sample mean, which can be arbitrarily away from the true mean due to the existence of contamination from $Q$.

Huber showed in his seminal work \citep{huber1965robust,huber1973minimax} that the exact optimal solution to the robust two-point testing problem is the following testing function:
$$\phi_{\text{Huber}}=\mathbb{I}\left\{\prod_{i=1}^n\left[\left(\frac{dP_1}{dP_0}(X_i)\vee c\right)\wedge C\right]>t\right\},$$
for some $0<c<C<\infty$. It can be seen as a clipped likelihood ratio test. By clipping the likelihood ratio functions that have enormous or infinitesimal values, the influence from outliers can be diminished. When $\epsilon=0$, the clipping cut-offs become $c=0$, $C=\infty$, and $\phi_{\text{Huber}}$ naturally reduces to the likelihood ratio test. Though $\phi_{\text{Huber}}$ exactly minimizes the testing error, the clipping cut-offs $c$ and $C$ depend on the knowledge of $\epsilon$, a quantity that characterizes the contamination proportion. Since it is impossible to estimate $\epsilon$ when $Q$ is not specified \citep{chen2015robust}, Huber's approach is not adaptive to the contamination proportion $\epsilon$ and thus not applicable.

Another work related to the robust testing problem is by Le Cam \citep{lecam1973convergence} and Birg\'{e} \citep{birge1979th}. Instead of testing between two $\epsilon$-contamination neighborhoods, they considered two Hellinger balls:
\begin{eqnarray*}
H_0: && P\in\left\{P: H(P,P_0)\leq \tau\right\}, \\
H_1: && P\in\left\{P: H(P,P_1)\leq\tau\right\}.
\end{eqnarray*}
They constructed a testing function and established the following testing error
\begin{eqnarray}
\nonumber &&\sup_{{P\in\left\{P: H(P,P_0)\leq \tau\right\}}}P\phi+\sup_{P\in\left\{P: H(P,P_1)\leq\tau\right\}}P(1-\phi)\\
\label{eq:error_lecam} &\leq& 2\exp\left(-\frac{n}{2}\left(H(P_0,P_1)-2\tau\right)^2\right),
\end{eqnarray}
for any $\tau < \frac{1}{2}H(P_0,P_1)$. However, their procedure cannot give optimal rate under Huber's setting. To put an $\epsilon$-contamination neighborhood into a $\tau$-Hellinger ball, the smallest $\tau$ would be $\sqrt{2\epsilon}$. That is,
$$\left\{(1-\epsilon)P_0+\epsilon Q: Q\right\}\subset\left\{P: H(P,P_0)\leq \sqrt{2\epsilon}\right\}.$$
When it comes to estimation, it will result in a sub-optimal $\epsilon$ term instead of the optimal $\epsilon^2$ in (\ref{eq:rate}).

In this paper, we propose a solution to the robust two-point testing problem as follows:
\begin{equation}
\phi=\mathbb{I}\left\{|P_n(A)-P_0(A)|>|P_n(A)-P_1(A)|\right\},\label{eq:test}
\end{equation}
where $P_n(\cdot)$ denotes the empirical distribution such that
$$P_n(A)=\frac{1}{n}\sum_{i=1}^n\mathbb{I}\{X_i\in A\}, $$
and $A$ is chosen as a measurable set that maximally distinguishes $P_0$ and $P_1$. That is,
\begin{equation}
A=\arg\max_A|P_0(A)-P_1(A)|=\{p_0>p_1\},\label{eq:A}
\end{equation}
where $p_j$ is the density function defined as $p_j=\frac{dP_j}{d(P_0+P_1)}$ for $j=0,1$. The corresponding estimator of the testing function $\phi$ is called Scheff{\'e} estimate by Devroye and Lugosi in their book \cite{devroye2012combinatorial} under the framework of density estimation. This is built on an $L_1$-based estimator proposed by Yatracos in \cite{yatracos1985rates}. The intuition is that with the set $A$ possessing maximal distinguishing power, we check whether the empirical probability of $A$ is closer to $P_0(A)$ or $P_1(A)$. Since we use summation of indicator functions to collect the information offered by each data point separately, compared to the product form taking by the likelihood ratio test, it is robust to outliers. Moreover, the proposed testing procedure does not depend on the contamination proportion $\epsilon$. The testing error of the proposed procedure is characterized by the following theorem.
\begin{thm}\label{thm:test}
Assume $\TV(P_0,P_1)>2\epsilon$. Then we have
\begin{eqnarray*}
&& \sup_{P\in\left\{(1-\epsilon)P_0+\epsilon Q: Q\right\}}P\phi+\sup_{P\in\left\{(1-\epsilon)P_1+\epsilon Q: Q\right\}}P(1-\phi) \\
&\leq& 4\exp\left(-\frac{1}{2}n\left(\TV(P_0,P_1)-2\epsilon\right)^2\right).
\end{eqnarray*}
\end{thm}
We emphasize that Theorem \ref{thm:test} says the exponent of the testing error is proportional to $n\left(\TV(P_0,P_1)-2\epsilon\right)^2$. Although Scheff{\'e} estimate was first proposed in \cite{devroye2012combinatorial} for density estimation problems, this important property on exponent of the testing error was not explicitly explored and thus is new. Compared with Le Cam and Birg\'{e}'s testing error (\ref{eq:error_lecam}), the exponent of ours is characterized by the total variation distance instead of the Hellinger distance. As we will show in Section \ref{sec:upper}, this exponent leads to minimax optimal estimation for Huber's $\epsilon$-contamination model.

%% file: upper_re.tex

\section{Construction of Upper Bounds}\label{sec:upper}

In this section, we present a general principle for the construction of a robust estimator given i.i.d. observations $X_1,...,X_n\sim (1-\epsilon)P_{\theta}+\epsilon Q$ with $\theta\in\Theta$ for some parameter space $\Theta$. We assume that the parameter space $\Theta$ is totally bounded. Define $m=\mathcal{M}(\delta,\Theta,\TV(\cdot,\cdot))$ to be the smallest number such that there exists $\{\theta_1,...,\theta_m\}\subset\Theta$ satisfying that for any $\theta\in\Theta$, there is a $j\in[m]$ such that $\TV(P_{\theta_j}, P_{\theta})\leq\delta$. We call $\{\theta_1,...,\theta_m\}\subset\Theta$ a $\delta$-covering set and $\mathcal{M}(\delta,\Theta,\TV(\cdot,\cdot))$ is the corresponding covering number.
The estimator of $\theta$ is constructed by performing robust testing (\ref{eq:test}) for each pair in the $\delta$-covering set and then selecting the most favorable one. To be specific, given i.i.d. observations, for any $j\neq k$, define the testing function
\begin{eqnarray*}
\phi_{jk} &=& \mathbb{I}\left\{\left|\frac{1}{n}\sum_{i=1}^n\mathbb{I}\{p_{\theta_j}(X_i)>p_{\theta_k}(X_i)\}-P_{\theta_j}(p_{\theta_j}(X)>p_{\theta_k}(X))\right|\right.\\
&&\left.>\left|\frac{1}{n}\sum_{i=1}^n\mathbb{I}\{p_{\theta_j}(X_i)>p_{\theta_k}(X_i)\}-P_{\theta_k}(p_{\theta_j}(X)>p_{\theta_k}(X))\right|\right\},
\end{eqnarray*}
where $p_{\theta}=\frac{dP_{\theta}}{d\mu}$ is the density function for some common dominating measure $\mu$. When $\phi_{jk}=1$, $\theta_k$ is favored over $\theta_j$. When $\phi_{jk}=0$, $\theta_j$ is favored over $\theta_k$. Finally, the robust estimator is defined as $\hat{\theta}=\theta_{\hat{j}}$ with
\begin{equation}
\hat{j}=\arg\min_{j\in[m]}\sum_{k\neq j}\phi_{jk}.\label{eq:estimator}
\end{equation}
That is to say, the final estimator wins the maximum number of pair-wise competitions. When (\ref{eq:estimator}) has multiple minimizers, $\hat{j}$ is understood to be any one of them. This estimator $\theta_{\hat{j}}$ is also called Scheff{\'e} tournament winner in \cite{devroye2012combinatorial} within the framework of density estimation. The detailed comparison will be discussed in Remark \ref{rem:comparison} later. Since the testing procedure introduced in Section \ref{sec:test} is adaptive for the contamination proportion $\epsilon$, the estimator (\ref{eq:estimator}) is also adaptive for $\epsilon$. The estimation error is upper bounded by the following theorem.
\begin{thm}\label{thm:estimator}
Assume $\eta>8(\epsilon+\delta)$. For the estimator $\hat{\theta}$ defined above, we have
\begin{eqnarray*}
&& \sup_{\theta\in\Theta,Q}\mathbb{P}_{(\epsilon,\theta,Q)}\left\{\TV(P_{\hat{\theta}},P_{\theta})>\eta+\delta\right\} \\
&\leq& 4\mathcal{M}^2(\delta,\Theta,\TV(\cdot,\cdot))\exp\left(-\frac{1}{2}n(\eta/4-2(\epsilon+\delta))^2\right),
\end{eqnarray*}
where the probability $\mathbb{P}_{(\epsilon,\theta,Q)}$ is defined in (\ref{eq:central}).
\end{thm}
The theorem immediately implies the convergence rate (\ref{eq:rate}) when we let
$$\eta^2=C\left[\left\{\frac{\log \mathcal{M}(\delta,\Theta,\TV(\cdot,\cdot))}{n}+\delta^2\right\}\vee\epsilon^2\right]$$
for some large constant $C$ and then minimize the rate over $\delta$. To show the rate (\ref{eq:rate}) implied by Theorem \ref{thm:estimator} is minimax optimal, we first review a general lower bound result in \cite{chen2015robust}.

\begin{thm}[Chen, Gao \& Ren (2015) \citep{chen2015robust}]\label{thm:lower}
$L(\cdot,\cdot)$ is a loss function defined on the parameter space $\Theta$. Define
$$\omega(\epsilon,\Theta)=\sup\left\{L(\theta_1,\theta_2): \TV(P_{\theta_1},P_{\theta_2})\leq \epsilon/(1-\epsilon); \theta_1,\theta_2\in\Theta\right\}.$$
Suppose there is some $\mathcal{R}(0)$ such that
\begin{equation}
\inf_{\hat{\theta}}\sup_{\theta\in\Theta,Q}\mathbb{P}_{(\epsilon,\theta,Q)}\left\{L(\hat{\theta},\theta)\geq\mathcal{R}(\epsilon)\right\}\geq c\label{eq:lower}
\end{equation}
holds for $\epsilon=0$. Then, (\ref{eq:lower}) holds for $\mathcal{R}(\epsilon)\asymp\mathcal{R}(0)\vee\omega(\epsilon,\Theta)$.
\end{thm}
Theorem \ref{thm:lower} provides a lower bound for general loss functions. The quantity $\omega(\epsilon,\Theta)$ is called modulus of continuity defined by Donoho and Liu \citep{donoho1991geometrizing,donoho1994statistical}. For total variation loss, $\omega(\epsilon,\Theta)\asymp \epsilon$. Moreover, a general lower bound result by Yang and Barron \citep{yang1999information} implies the formula
$$\mathcal{R}^2(0)\asymp \min_{\delta>0}\left\{\frac{\log \mathcal{M}(\delta,\Theta,\TV(\cdot,\cdot))}{n}+\delta^2\right\},$$
under very mild conditions. Hence, (\ref{eq:rate}) is also the minimax lower bound for the problem.

\begin{remark}
Both Theorem \ref{thm:estimator} and Theorem \ref{thm:lower} are stated in probability. To obtain the same conclusion in expectation as defined by (\ref{eq:minimax}), observe that the in-probability lower bound directly implies an in-expectation lower bound via Markov inequality. The in-expectation upper bound can be calculated by integrating over the tail probability of Theorem \ref{thm:estimator}.
\end{remark}

For some parametric and high-dimensional models, the notion of global covering number may not provide a tight upper bound. We show an improvement of Theorem \ref{thm:estimator} by using the notion of local covering number. Let $\Theta'=\{\theta_1,...,\theta_m\}$ be a $\delta$-covering set for $\Theta$. For any integer $l$, define
$$D_l(\delta)=\max_{\theta_0\in\Theta'}\left|\left\{\theta\in\Theta': l\delta<\TV(P_{\theta},P_{\theta_0})\leq (l+1)\delta\right\}\right|.$$

\begin{thm}\label{thm:local}
Let $L$ be any number such that $\frac{L}{4}\delta-2\epsilon-2\delta>0$ and $L/4$ is an integer. For the estimator $\hat{\theta}$ defined by (\ref{eq:estimator}), we have
\begin{eqnarray*}
&& \sup_{\theta\in\Theta,Q}\mathbb{P}_{(\epsilon,\theta,Q)}\left\{\TV(P_{\hat{\theta}},P_{\theta})>(L+1)\delta\right\} \\
&\leq& 2\sum_{l\geq L/4}D_l(\delta)\exp\left(-\frac{1}{2}n\left(l\delta-2(\epsilon+\delta)\right)^2\right) \\
&& + 2\left[\sum_{l=0}^{L/4-1}D_l(\delta)\right]\sum_{l\geq L}D_l(\delta)\exp\left(-\frac{1}{2}n\left((l-3L/4)\delta-2(\epsilon+\delta)\right)^2\right),
\end{eqnarray*}
where the probability $\mathbb{P}_{(\epsilon,\theta,Q)}$ is defined in (\ref{eq:central}).
\end{thm}

\begin{remark}\label{rem:comparison}
A closely related estimator called minimum distance estimator was first
proposed by Yatracos in \cite{yatracos1985rates}. Later, built on Yatracos'
method, Devroye and Lugosi in their book \cite{devroye2012combinatorial}
further proposed a similar estimator called Scheff{\'e} tournament winner as the one in (\ref{eq:estimator}) within the
framework of density estimation. While the theoretical decision framework
under the Huber's $\epsilon $-contamination model was not considered in
either \cite{yatracos1985rates} or \cite{devroye2012combinatorial}, it is
worthwhile to point out another subtle while essential difference between
the results in those early works and the current paper. Since the
analysis of main results in \cite{yatracos1985rates} and \cite%
{devroye2012combinatorial} are similar, we only focus on the analysis for the
minimum distance estimator in this paper. The minimum distance estimator $%
\hat{\theta}^{Y}=\theta _{\hat{\jmath}}$ is defined as
\begin{equation*}
\hat{\jmath}=\arg \min_{j\in \lbrack m]}\sup_{A\in \mathcal{A}}\left\vert
P_{n}(A)-P_{\theta _{j}}(A)\right\vert ,
\end{equation*}%
where the Yatracos class $\mathcal{A}=\{\{p_{\theta _{i}}>p_{\theta
_{j}}\}:i\neq j\in \lbrack m]\}$ is the collection of the sets $A$ in (\ref%
{eq:test}) applied on each pair of distributions indexed by the $\delta $%
-covering set $\{\theta _{1},\ldots ,\theta _{m}\}$. This estimator, instead
of being explicitly built on pair-wise competitions, minimizes the distance to the
empirical measure uniformly over the Yatracos class $\mathcal{A}$.
Consequently, it is unlikely for $\hat{\theta}^{Y}$ to take advantage of
the more delicate local covering set in a layer-by-layer fashion as stated
in Theorem \ref{thm:local} for our analysis. In this sense, the optimality cannot be
achieved for several high-dimensional models. See Sections \ref{sec:reg} and \ref{sec:rank} for
two such examples.
\end{remark}

%% file: application_re.tex

\newcommand{\floor}[1]{{\left\lfloor {#1} \right \rfloor}}

\section{Applications}\label{sec:application}

To illustrate the theorems in Section \ref{sec:upper}, here we present their applications on three problems: density estimation with H\"{o}lder smoothness, sparse linear regression and low-rank trace regression.

\subsection{Density Estimation}

Consider i.i.d. observation $X_1,...,X_n\sim \mathbb{P}_{(\epsilon,f,Q)}=(1-\epsilon)P_f+\epsilon Q$, where $f=\frac{dP}{d\lambda}$ is the density function of $P_f$ supported on $[0,1]$ with respect to the Lebesgue measure. We consider the H\"{o}lder class for the density function on $[0,1]$. Let $\{\phi_{lk}\}_{l\geq 0,0\leq k\leq 2^l-1}$ be an orthogonal wavelet basis on the interval $[0,1]$. The precise construction of the wavelet basis is referred to \cite{cohen1993wavelets}.
Define the following density class:
\begin{eqnarray}
\label{eq:holder} && \mathcal{H}_{den}(\beta,M) \\
\nonumber &=& \left\{f=\sum_{\substack{l\geq 0,\\ 0\leq k\leq 2^l-1}}f_{lk}\psi_{lk}: f\geq 0, \int_0^1f=1,\sup_{\substack{l\geq 0,\\ 0\leq k\leq 2^l-1}}2^{l(1/2+\beta)}|f_{lk}|\leq M\right\},
\end{eqnarray}
where $\beta>0$ is the smoothness index of the function class. The constant $M>0$ is the radius of the class. By \cite{yang1999information},
$$\log\mathcal{M}\left(\delta,\mathcal{H}_{den}(\beta,M), \TV(\cdot,\cdot)\right)\asymp \delta^{-1/\beta}.$$
Therefore, using the estimator (\ref{eq:estimator}) with $\delta\asymp n^{-\frac{\beta}{2\beta+1}}$, Theorem \ref{thm:estimator} implies the following convergence rate.

\begin{corollary}\label{cor:density}
For the H\"{o}lder class $\mathcal{H}_{den}(\beta,M)$, there are some constants $C,C'$, such that
$$\norm{\hat{f}-f}_1^2\leq C\left(n^{-\frac{2\beta}{2\beta+1}}\vee\epsilon^2\right),$$
with $\mathbb{P}_{(\epsilon,f,Q)}$-probability at least $1-\exp\left(-C'\left(n^{\frac{1}{2\beta+1}}+n\epsilon^2\right)\right)$ uniformly over all $Q$ and $f\in \mathcal{H}_{den}(\beta,M)$.
\end{corollary}

Given the equation $\TV(P_{f_1},P_{f_2})=\frac{1}{2}\norm{f_1-f_2}_1$, Corollary \ref{cor:density} states the convergence result in the squared $\ell_1$ distance.
Combining with Theorem \ref{thm:lower} and the discussion thereafter, which implies $n^{-\frac{2\beta}{2\beta+1}}\vee\epsilon^2$ is also the minimax lower bound, we conclude it is the minimax rate for this problem. When $\epsilon^2\lesssim n^{-\frac{2\beta}{2\beta+1}}$, the rate is dominated by $n^{-\frac{2\beta}{2\beta+1}}$. This is the minimax rate for density estimation when there is no contamination. When $n^{-\frac{2\beta}{2\beta+1}}\lesssim\epsilon^2$, the rate is dominated by $\epsilon^2$. Therefore, the maximum expected number of outliers that can be tolerated without breaking down the usual minimax rate is $n\epsilon\asymp n^{\frac{\beta+1}{2\beta+1}}$.

\subsection{Sparse Linear Regression}\label{sec:reg}

For the linear regression model, we consider a random design setting
$$y_i=X_i^T\theta+z_i,$$
where without contamination $X_i\sim N(0,\Sigma)$ and $z_i\sim N(0,\sigma^2)$ are independent. Under the current setting, both the design and the response in the model can be contaminated. That is, we have i.i.d. observations $(X_1,y_1),...,(X_n,y_n)\sim \mathbb{P}_{(\epsilon,\theta,Q)}=(1-\epsilon)P_\theta+\epsilon Q$, where the $P_{\theta}$ denotes the probability distribution of
$$p(X,y)=p(X)p(y|X),$$
with $p(X)=N(0,\Sigma)$ and $p(y|X)=N(X^T\theta,\sigma^2)$. 

Given the covariance matrix $\Sigma$ of $X$, we further impose the \textit{sparse eigenvalue conditions} as follows, 
\begin{eqnarray}
\inf_{|\supp(v)|=2s}\norm{\Sigma^{1/2}v}/\norm{v}\geq \kappa, \label{eq:lowerRE}\\ \sup_{|\supp(v)|=2s}\norm{\Sigma^{1/2}v}/\norm{v}\leq \kappa_u. \label{eq:upperRE}
\end{eqnarray}
In addition, we assume $\kappa_u \asymp \kappa$. In other words, the upper and lower sparse eigenvalues are at the same order, which is satisfied, for example, if all eigenvalues of $\Sigma$ are at the same order. Given noise level $\sigma$ and sparse eigenvalue level $\kappa$, we consider the following sparse set as the  parameter space for $\theta$:
$$\Theta(s,M,\sigma,\kappa)=\left\{\theta\in\mathbb{R}^p: |\supp(\theta)|\leq s, \norm{\theta}\leq M\sigma/\kappa\right\},$$
where $s>0$ is the sparsity of the regression coefficients and $M>0$ is assumed to be a constant.

\begin{remark}
	The total variation distance $\TV(P_{\theta},P_{\theta'})$ is upper bounded by $C\|\Sigma^{1/2}(\theta-\theta')\|/\sigma$ with some constant $C>0$. Therefore, we impose an upper bound  $\norm{\theta}\leq M\sigma/\kappa$ for the parameter $\theta$ when defining parameter space $\Theta(s,M,\sigma,\kappa)$ to guarantee that the parameter space is totally bounded under the loss $\TV(\cdot,\cdot)$. This is a natural condition and is assumed at the beginning of Section \ref{sec:upper}. For this totally bounded parameter space $\Theta(s,M,\sigma,\kappa)$, the equivalence of $\TV(P_{\theta},P_{\theta'})$ and $\|\Sigma^{1/2}(\theta-\theta')\|/\sigma$ can be further established. See Lemma \ref{lem:TV} for details.
\end{remark}

For this set, we will show that
$$\log D_l(\delta)\lesssim s\log\frac{ep}{s}+s\log(l+1).$$
Then, using the estimator (\ref{eq:estimator}) with $\delta\asymp \sqrt{\frac{s\log\frac{ep}{s}}{n}}$, Theorem \ref{thm:local} implies the following convergence rate.

\begin{corollary}\label{cor:reg}
We assume $s\log\frac{ep}{s}\leq cn$ with some sufficiently small $c>0$.  Then, there are some constants $C,C'$, such that
\begin{eqnarray*}
\norm{\Sigma^{1/2}(\hat{\theta}-\theta)}^2\leq C\sigma^2\left(\frac{s\log\frac{ep}{s}}{n}\vee\epsilon^2\right) \\
\norm{\hat{\theta}-\theta}^2\leq C\frac{\sigma^2}{\kappa^2}\left(\frac{s\log\frac{ep}{s}}{n}\vee\epsilon^2\right),
\end{eqnarray*}
with $\mathbb{P}_{(\epsilon,\theta,Q)}$-probability at least $1-\exp\left(-C'\left(s\log\frac{ep}{s}+n\epsilon^2\right)\right)$ uniformly over $\theta\in\Theta(s,M,\sigma,\kappa)$ and all $Q$.
\end{corollary}

We use Theorem \ref{thm:local} instead of Theorem \ref{thm:estimator} to derive Corollary \ref{cor:reg}, because Theorem \ref{thm:estimator} uses global metric entropy and will cause an extra logarithmic factor in the convergence rate.
For the prediction error loss $\norm{\Sigma^{1/2}(\hat{\theta}-\theta)}^2$, the rate does not depend on the covariance $\Sigma$ of the design matrix. On the other hand, the rate for the estimation error loss $\norm{\hat{\theta}-\theta}^2$ depends on the sparse eigenvalue $\kappa$ of $\Sigma$.

When $\epsilon=0$, both rates in Corollary \ref{cor:reg} are known to be minimax optimal. Indeed, the lower bounds with explicit dependence on $\kappa$ and $\sigma$ can be found in Theorem 1(b) and Theorem 3(b) of \cite{raskutti2011minimax}, by observing $\kappa_u \asymp \kappa$. In particular, one can easily check that the least favorable subset in the lower bound construction (Proof of Theorem 1(b) in \cite{raskutti2011minimax}) is contained in our parameter space $\Theta(s,M,\sigma,\kappa)$ for any fixed $M$. We emphasize that although an empirical version of sparse eigenvalue conditions of (\ref{eq:lowerRE})-(\ref{eq:upperRE}) is used in Assumption 3 of \cite{raskutti2011minimax}, it is well known that (see, for example, \cite{rudelson2013reconstruction}) under the assumption $s\log\frac{ep}{s}\leq cn$ with some sufficiently small $c>0$ and $\kappa_u \asymp \kappa$, our population sparse eigenvalue conditions (\ref{eq:lowerRE})-(\ref{eq:upperRE}) imply the empirical version with values at the same order of $\kappa$ and $\kappa_u$ with probability at least $1-\exp(-Cn)$. For $\epsilon>0$, due to the equivalence of total variation distance $\TV(P_{\theta},P_{\theta'})$ and $\|\Sigma^{1/2}(\theta-\theta')\|/\sigma$ as well as $\kappa_u \asymp \kappa$, the modulus of continuity for the prediction error loss scales as $\omega(\epsilon,\Theta)\asymp \sigma\epsilon$, while for the estimation error loss, it scales as $\omega(\epsilon,\Theta)\asymp \sigma\epsilon/\kappa$. Hence, by Theorem \ref{thm:lower}, both rates in Corollary \ref{cor:reg} are minimax optimal.

\subsection{Low-Rank Trace Regression}\label{sec:rank}

Consider the observation pair $(X_i,y_i)$ satisfying the model
$$y_i=\Tr(X_i^TA)+z_i,$$
where $X_i\in\mathbb{R}^{p_1\times p_2}$ is an observed design matrix and $A\in\mathbb{R}^{p_1\times p_2}$ is an unknown low-rank signal matrix. The problem of recovering a high-dimensional low-rank matrix has been considered in \cite{recht2010guaranteed,candes2011tight,rohde2011estimation,koltchinskii2011nuclear}. However, these results all assume the data are generated without contamination. In many practical situations, both the design and the response can be contaminated.
For some covariance matrix $\Sigma\in\mathbb{R}^{p_1p_2\times p_1p_2}$ and some number $\sigma>0$, we assume i.i.d. observations $(X_1,y_1),...,(X_n,y_n)\sim \mathbb{P}_{(\epsilon,A,Q)}=(1-\epsilon)P_{A}+\epsilon Q$, where $P_A$ denotes the probability distribution
$$p(X,y)=P(X)P(y|X),$$
with $p(X)$ referring to $\text{vec}(X)\sim N(0,\Sigma)$ and $p(y|X)=N(\Tr(X^TA),\sigma^2)$.

Given the covariance matrix $\Sigma$ of $\text{vec}(X)$, we further impose the \textit{restricted isometry condition} as follows, 
\begin{equation}
\kappa \leq \inf_{\rank(A)\leq 2r}\frac{\norm{\Sigma^{1/2}\text{vec}(A)}}{\fnorm{A}} \leq \sup_{\rank(A)\leq 2r}\frac{\norm{\Sigma^{1/2}\text{vec}(A)}}{\fnorm{A}} \leq \kappa_u, \label{eq:RI}
\end{equation}
In addition, we assume $\kappa_u \asymp \kappa$. A special case would be that all eigenvalues of $\Sigma$ are at the same order. Given noise level $\sigma$ and $\kappa$ from the restricted isometry condition, we assume the coefficient matrix $A$ is in a low-rank matrix class defined as
$$\mathcal{A}(r,M,\sigma,\kappa)=\left\{A\in\mathbb{R}^{p_1\times p_2}: \rank{(A)}\leq r,\fnorm{A}\leq M\sigma/\kappa\right\}.$$
The number $r>0$ upper bounds the rank. We assume $M$ is a constant throughout this section. 

\begin{remark} \label{rem:low rank}
	The total variation distance $\TV(P_A,P_{A'})$ is upper bounded by $C\|\Sigma^{1/2}(\text{vec}(A)-\text{vec}(A'))\|/\sigma$ with some constant $C>0$. Consequently, we impose an upper bound $\fnorm{A}\leq M\sigma/\kappa$ in parameter space $\mathcal{A}(r,M,\sigma,\kappa)$ to guarantee the parameter space is totally bounded under the loss $\TV(P_A,P_{A'})$. Similar to the setting of sparse linear regression, the equivalence of $\TV(P_A,P_{A'})$ and  $\|\Sigma^{1/2}(\text{vec}(A)-\text{vec}(A'))\|/\sigma$ can be further established for $\mathcal{A}(r,M,\sigma,\kappa)$ in the low-rank trace regression setting. See Lemma \ref{lem:TV} for details.
\end{remark}

For this low-rank matrix class, we will show that
$$\log D_l(\delta)\lesssim r(p_1+p_2)\log(l+1).$$
Then, for the estimator (\ref{eq:estimator}) with $\delta\asymp\sqrt{\frac{r(p_1+p_2)}{n}}$, Theorem \ref{thm:local} implies the following convergence rate.

\begin{corollary}\label{cor:rank}
Assume $r(p_1+p_2)\leq cn$ with some sufficiently small $c>0$. Then, there are constants $C,C'$, such that
$$\fnorm{\hat{A}-A}^2\leq C\frac{\sigma^2}{\kappa^2}\left(\frac{r(p_1+p_2)}{n}\vee\epsilon^2\right),$$
with $\mathbb{P}_{(\epsilon,A,Q)}$-probability at least $1-\exp\left(-C'\left(r(p_1+p_2)+n\epsilon^2\right)\right)$ uniformly over $A\in\mathcal{A}(r,M,\sigma,\kappa)$ and all $Q$.
\end{corollary}

The rate consists of two parts. The first part is the usual low-rank matrix estimation rate $\frac{\sigma^2r(p_1+p_2)}{\kappa^2n}$, which is known to be minimax optimal with explicit dependence on both $\sigma$ and $\kappa$ when $\epsilon=0$. See, for example, Theorem 5 of \citep{koltchinskii2011nuclear}. To interpret this lower bound in \citep{koltchinskii2011nuclear}, we emphasize that a similar restricted isometry condition as in (\ref{eq:RI}) is imposed in Assumption 2 of \citep{koltchinskii2011nuclear} with $\mu \asymp \kappa^{-1}$ and $\|A\|_{L_2(\Pi)} \asymp \|\Sigma^{1/2}\text{vec}(A)\|$ in our setting respectively. In addition, it is easy to calculate that the least favorable subset $\mathcal{B(C)}$ in the construction of lower bound in \citep{koltchinskii2011nuclear} is indeed contained in our parameter space with any fixed $M$, due to the condition that $r(p_1+p_2)\leq cn$ with some sufficiently small $c>0$. The second part is $\frac{\sigma^2\epsilon^2}{\kappa^2}$, which is contributed by the modulus of continuity $\omega^2(\epsilon,\mathcal{A})$ for this problem, noting that $\kappa_u \asymp \kappa$ in (\ref{eq:RI}) and the equivalence of $\TV(P_A,P_{A'})$ and  $\|\Sigma^{1/2}(\text{vec}(A)-\text{vec}(A'))\|/\sigma$ in Remark \ref{rem:low rank}. Therefore, by Theorem \ref{thm:lower}, the upper bound in Corollary \ref{cor:rank} is minimax optimal.

%% file: sup-norm_re.tex

\section{Discussion with an Example under Supreme Norm}\label{sec:sup}

This paper gives a general framework to construct robust estimators under Huber's $\epsilon$-contamination model. The key idea of the construction lies in the robust testing procedure Scheff{\'e} estimate. We emphasize that this robust testing procedure enjoys a desired error exponent that depends on the total variation distance, which is intrinsic to Huber's robust setting. This new result is stated precisely in Theorem \ref{thm:test}. Consequently, the rate-optimal estimators that we present in Section \ref{sec:application} all depend on the general theorems in Section \ref{sec:upper} under loss functions that are equivalent to the total variation distance. However, it is unknown whether the theory can be extended to some important loss functions that are not equivalent to the total variation distance. In this section, we give an example for a supreme norm loss function in the context of a nonparametric white noise model. We show that the minimax rate of the problem depends on the contamination proportion in a different way. The general treatment for non-intrinsic loss functions will be considered as future projects.

The white noise model \citep{pinsker1980optimal} is considered to be a standard nonparametric model for function estimation \citep{brown1996asymptotic,nussbaum1996asymptotic}. By observing the stochastic process
\begin{equation}
dY_t=f(t)dt+\frac{1}{\sqrt{n}}dW_t,\quad t\in[0,1],\label{eq:white}
\end{equation}
with a standard Wiener process $\{W_t\}_{t\in[0,1]}$, the goal is to estimate the function $f$. Equivalently, (\ref{eq:white}) can be written as an i.i.d. model. That is, we observe i.i.d. stochastic processes $\{Y_{t,1}\}_{t\in[0,1]},...,\{Y_{t,n}\}_{t\in[0,1]}\sim P_f$, where $P_f$ denotes the probability distribution
\begin{equation}
dY_{t,i}=f(t)dt+dW_{t,i},\label{eq:white_1}
\end{equation}
Under Huber's framework, there is an $\epsilon$ probability of contamination, and we observe i.i.d. stochastic processes $\{Y_{t,1}\}_{t\in[0,1]},...,\{Y_{t,n}\}_{t\in[0,1]}\sim \mathbb{P}_{(\epsilon,f,Q)}=(1-\epsilon)P_f+\epsilon Q$. We use a slightly modified version of H\"{o}lder class defined in (\ref{eq:holder}):
$$
\mathcal{H}(\beta,M)=\left\{f=\sum_{l\geq 0, 0\leq k\leq 2^l-1}f_{lk}\psi_{lk}:\sup_{l\geq 0,0\leq k\leq 2^l-1}2^{l(1/2+\beta)}|f_{lk}|\leq M\right\},
$$
where $\{\phi_{lk}\}_{l\geq 0,0\leq k\leq 2^l-1}$ is an orthogonal wavelet basis on the interval $[0,1]$, see \cite{cohen1993wavelets} for the detailed construction.

We are going to construct an estimator that achieves the optimal rate under the supreme loss $\norm{\hat{f}-f}_{\infty}$. Let $L$ be the largest integer such that $2^L\leq\left(\frac{\log n}{n}\vee\epsilon^2\right)^{-\frac{1}{2\beta+1}}$. The estimator is $\hat{f}=\sum_{0\leq l\leq L,0\leq k\leq 2^l-1}\hat{f}_{lk}\psi_{lk}$ for
$$\hat{f}_{lk}=\text{Median}\left(\left\{y_{lk,i}\right\}_{i=1}^n\right),$$
where $y_{lk,i}=\int_0^1\psi_{lk}(t)dY_{t,i}$ are empirical wavelet coefficients.

\begin{thm}\label{thm:sup}
Assume $\epsilon<1/4$.
For the H\"{o}lder class $\mathcal{H}(\beta,M)$, there are constants $C,C'$, such that
$$\norm{\hat{f}-f}_{\infty}^2\leq C\left[\left(\frac{n}{\log n}\right)^{-\frac{2\beta}{2\beta+1}}\vee \epsilon^{\frac{4\beta}{2\beta+1}}\right],$$
with $\mathbb{P}_{(\epsilon,f,Q)}$-probability at least $1-\exp\left(-C'\left(\log n + n\epsilon^2\right)\right)$ uniformly over $f\in \mathcal{H}(\beta,M)$ and all $Q$.
\end{thm}

This theorem characterizes the upper bound of this problem. By applying Theorem \ref{thm:lower}, we show it is also the minimax lower bound.

\begin{corollary}\label{cor:sup}
There are some constants $C,c>0$ such that
$$\inf_{\hat{f}}\sup_{f\in\mathcal{H}(\beta,M), Q}\mathbb{P}_{(\epsilon,f,Q)}\left\{\norm{\hat{f}-f}_{\infty}^2> C\left[\left(\frac{n}{\log n}\right)^{-\frac{2\beta}{2\beta+1}}\vee \epsilon^{\frac{4\beta}{2\beta+1}}\right]\right\}>c.$$
\end{corollary}

Combining Theorem \ref{thm:sup} and Corollary \ref{cor:sup}, we conclude that $\left(\frac{n}{\log n}\right)^{-\frac{2\beta}{2\beta+1}}\vee \epsilon^{\frac{4\beta}{2\beta+1}}$ is the minimax rate for estimating a nonparametric drift function $f$ under the supreme loss in Huber's framework. Compared with Corollary \ref{cor:density}, the dependence on the contamination proportion is through $\epsilon^{\frac{4\beta}{2\beta+1}}$ instead of the usual $\epsilon^2$ for the total variation loss. This is because for the supreme loss, $\epsilon^{\frac{2\beta}{2\beta+1}}$ is the modulus of continuity defined in Theorem \ref{thm:lower}.
When $\epsilon=0$, the rate reduces to the usual nonparametric rate for supreme loss \citep{tsybakov2008introduction}.

\begin{remark}
Note that the estimator $\hat{f}$ does not use the general construction in Section \ref{sec:upper}. As a consequence, it requires the knowledge of the contamination proportion $\epsilon$. However, it reveals a minimax rate with an interesting dependence on $\epsilon$, which is different from the rates of the estimators in Section \ref{sec:upper} and Section \ref{sec:application}. It is of great interest to us how to construct an estimator that is adaptive to $\epsilon$ for the supreme loss. A more general open question is to seek ways of construction of estimators for other non-intrinsic loss functions.
\end{remark}

%% file: proof_re.tex

\section{Proofs}\label{sec:proof}

This section collects the proofs of all technical results in the paper. The proofs of the results in Section \ref{sec:test} and Section \ref{sec:upper} are given in Section \ref{sec:6.1}. The proofs of the results in Section \ref{sec:application} and Section \ref{sec:sup} are given in Section \ref{sec:6.2} and Section \ref{sec:6.3}, respectively.

\subsection{Proofs in Section \ref{sec:test} and Section \ref{sec:upper}}\label{sec:6.1}

Before stating the proofs of the main theorems, we need the following lemma to upper bound the testing error with respect to distributions in a total variation neighborhood.
\begin{lemma}\label{lem:basic}
Consider the testing function $\phi$ in the form of (\ref{eq:test}). Assume $\TV(P_0,P_1)>2\xi$, and then
\begin{eqnarray*}
\sup_{\{P: \TV(P,P_0)\leq \xi\}}P\phi &\leq& 2\exp\left(-\frac{1}{2}n\left(\TV(P_0,P_1)-2\xi\right)^2\right), \\
\sup_{\{P: \TV(P,P_1)\leq \xi\}}P(1-\phi) &\leq& 2\exp\left(-\frac{1}{2}n\left(\TV(P_0,P_1)-2\xi\right)^2\right).
\end{eqnarray*}
\end{lemma}
\begin{proof}
Since the proofs of the two inequalities are the same, we only give details for the first one.
For any $P$ such that $\TV(P,P_0)\leq\xi$, we have
\begin{eqnarray}
\nonumber P\phi &=& P\left\{|P_n(A)-P_0(A)|>|P_n(A)-P_1(A)|\right\} \\
\label{eq:pf1} &\leq& P\left\{|P_n(A)-P_0(A)|>|P_0(A)-P_1(A)|-|P_n(A)-P_0(A)|\right\} \\
\label{eq:pf2} &=& P\left\{2|P_n(A)-P_0(A)|>\TV(P_0,P_1)\right\} \\
\label{eq:pf3} &\leq& P\left\{2|P_n(A)-P(A)|>\TV(P_0,P_1)-2\xi\right\} \\
\label{eq:pf4} &\leq& 2\exp\left(-\frac{1}{2}n\left(\TV(P_0,P_1)-2\xi\right)^2\right).
\end{eqnarray}
The inequality (\ref{eq:pf1}) is due to triangle inequality. By rearrangement and the definition of total variation distance, we get (\ref{eq:pf2}). Then, (\ref{eq:pf3}) is obtained through triangle inequality and the fact that $|P(A)-P_0(A)|\leq \TV(P,P_0)\leq\xi$. Finally, (\ref{eq:pf4}) is by Hoeffding's inequality. Taking supreme over the set $\{P: \TV(P,P_0)\leq \xi\}$, the proof is complete.
\end{proof}

Now we are ready to give the proofs of the main theorems.

\begin{proof}[Proof of Theorem \ref{thm:test}]
Note that
$$\{(1-\epsilon)P_0+\epsilon Q:Q\}\subset\{P:\TV(P,P_0)\leq\epsilon\},$$
and
$$\{(1-\epsilon)P_1+\epsilon Q:Q\}\subset\{P:\TV(P,P_1)\leq\epsilon\}.$$
Thus, the proof is complete.
\end{proof}

\begin{proof}[Proof of Theorem \ref{thm:estimator}]
Let us use the notation $\phi_j=\sum_{k\neq j}\phi_{jk}$ and $\Theta_j=\{\theta\in\Theta:\TV(P_{\theta},P_{\theta_j})\leq \delta\}$. For some $c\in(0,1)$, let $\mathcal{N}_j=\{k\neq j:\TV(P_{\theta_k},P_{\theta_j})\leq c\eta\}$. Then, for $P=(1-\epsilon)P_{\theta}+\epsilon Q$ with any $\theta\in\Theta_j$ and any $Q$, we have
\begin{eqnarray}
\nonumber && P\left\{\TV(P_{\hat{\theta}},P_{\theta})>\eta+\delta\right\} \\
\label{eq:pf5} &\leq& P\left\{\TV(P_{\theta_{\hat{j}}},P_{\theta_j})>\eta\right\} \\
\label{eq:pf6} &\leq& P\left\{\phi_j\geq\min_{\{k:\TV(P_{\theta_k},P_{\theta_j})>\eta\}}\phi_k\right\} \\
\label{eq:pf7} &\leq& P\left\{\phi_j>|\mathcal{N}_j|\right\} + P\left\{\min_{\{k:\TV(P_{\theta_k},P_{\theta_j})>\eta\}}\phi_k<|\mathcal{N}_j|+1\right\} \\
\nonumber &\leq& P\left\{\phi_{jk}=1\text{ for some }k\notin\mathcal{N}_j\right\} \\
\nonumber && + \sum_{\{k: \TV(P_{\theta_k},P_{\theta_j})>\eta\}}P\left\{\phi_{kl}=0\text{ for some }l\in\mathcal{N}_j\cup\{j\}\right\} \\
\label{eq:pf7.5} &\leq& \sum_{k\notin\mathcal{N}_j}P\phi_{jk} + \sum_{\{k: \TV(P_{\theta_k},P_{\theta_j})>\eta\}}\sum_{l\in\mathcal{N}_j\cup\{j\}}P(1-\phi_{kl}) \\
\label{eq:pf8} &\leq& 2\mathcal{M}(\delta,\Theta,\TV(\cdot,\cdot))\exp\left(-\frac{1}{2}n\left(c\eta-2(\epsilon+\delta)\right)^2\right) \\
\nonumber && + 2\mathcal{M}^2(\delta,\Theta,\TV(\cdot,\cdot))\exp\left(-\frac{1}{2}n\left((1-c)\eta-2(\epsilon+\delta+c\eta)\right)^2\right).
\end{eqnarray}
The inequality (\ref{eq:pf5}) is by $\theta\in\Theta_j$. Suppose $\phi_j<\min_{\{k:\TV(P_{\theta_k},P_{\theta_j})>\eta\}}\phi_k$, we must have $\TV(P_{\theta_{\hat{j}}},P_{\theta_j})\leq \eta$ by the definition of $\hat{j}$ in (\ref{eq:estimator}). Therefore,
$$\left\{\TV(P_{\theta_{\hat{j}}},P_{\theta_j})>\eta\right\}\subset\left\{\phi_j\geq\min_{\{k:\TV(P_{\theta_k},P_{\theta_j})>\eta\}}\phi_k\right\},$$
which implies (\ref{eq:pf6}). The inequality (\ref{eq:pf7}) uses the fact that $\{x\geq y\}\subset \{x>z\}\cup\{y<z+1\}$. Finally, (\ref{eq:pf8}) is obtained by applying Lemma \ref{lem:basic} with the relations
$$\{(1-\epsilon)P_{\theta}+\epsilon Q:\theta\in\Theta_j,Q\}\subset\{P:\TV(P,P_{\theta_j})\leq \epsilon+\delta\},$$
and
$$\{(1-\epsilon)P_{\theta}+\epsilon Q:\theta\in\Theta_j,Q\}\subset\{P:\TV(P,P_{\theta_l})\leq \epsilon+\delta+c\eta\}.$$
The proof is complete by choosing $c=\frac{1}{4}$.
\end{proof}

\begin{proof}[Proof of Theorem \ref{thm:local}]
Let us use the notation $\phi_j=\sum_{k\neq j}\phi_{jk}$ and $\Theta_j=\{\theta\in\Theta:\TV(P_{\theta},P_{\theta_j})\leq \delta\}$. For some $c\in(0,1)$, let $\mathcal{N}_j=\{k\neq j:\TV(P_{\theta_k},P_{\theta_j})\leq L\delta/4\}$. Then, for $P=(1-\epsilon)P_{\theta}+\epsilon Q$ with any $\theta\in\Theta_j$ and any $Q$, we have
\begin{eqnarray*}
&& P\left\{\TV(P_{\hat{\theta}},P_{\theta})>(L+1)\delta\right\} \\
&\leq& \sum_{\{k:\TV(P_{\theta_k},P_{\theta_j})>\frac{L\delta}{4}\}}P\phi_{jk} + \sum_{\{k:\TV(P_{\theta_k},P_{\theta_j})>L\delta\}}\sum_{\{t: \TV(P_{\theta_t},P_{\theta_j})\leq \frac{L\delta}{4}\}}P(1-\phi_{kt}).
\end{eqnarray*}
This is by the same argument for deriving (\ref{eq:pf7.5}) in the proof of Theorem \ref{thm:estimator}. Then, we have
\begin{eqnarray*}
&& \sum_{\{k:\TV(P_{\theta_k},P_{\theta_j})>L\delta/4\}}P\phi_{jk} \\
&\leq& \sum_{l\geq L/4}\sum_{\{k: l\delta<\TV(P_{\theta_k},P_{\theta_j})\leq (l+1)\delta\}}P\phi_{jk} \\
&\leq& 2\sum_{l\geq L/4}D_l(\delta)\exp\left(-\frac{1}{2}n(l\delta-2(\epsilon+\delta)^2)\right),
\end{eqnarray*}
where the last inequality is by
\begin{equation}
|\{k: l\delta<\TV(P_{\theta_k},P_{\theta_j})\leq (l+1)\delta\}|\leq D_l(\delta),\label{eq:used}
\end{equation}
 and Lemma \ref{lem:basic} with the relation
$$\{(1-\epsilon)P_{\theta}+\epsilon Q: \theta\in\Theta_j,Q\}\subset\{P:\TV(P,P_{\theta_j})\leq\epsilon+\delta\}.$$
We also have
\begin{eqnarray*}
&& \sum_{\{k:\TV(P_{\theta_k},P_{\theta_j})>L\delta\}}\sum_{\{t: \TV(P_{\theta_t},P_{\theta_j})\leq L\delta/4\}}P(1-\phi_{kt}) \\
&\leq& \sum_{l\geq L}\sum_{\{k:l\delta<\TV(P_{\theta_k},P_{\theta_j})\leq (l+1)\delta\}}\sum_{\{t: \TV(P_{\theta_t},P_{\theta_j})\leq L\delta/4\}}P(1-\phi_{kt}) \\
&\leq& 2\left[\sum_{l=0}^{L/4-1}D_l(\delta)\right]\sum_{l\geq L}D_l(\delta)\exp\left(-\frac{1}{2}n\left(l\delta-L\delta/4-2(\epsilon+\delta+L\delta/4)\right)^2\right),
\end{eqnarray*}
where the last inequality follows from (\ref{eq:used}),
$$\left|\{t\neq j: \TV(P_{\theta_t},P_{\theta_j})\leq L\delta/4\}\right|\leq \sum_{l=0}^{L/4-1}D_l(\delta),$$
and Lemma \ref{lem:basic} with the relations
$$\{(1-\epsilon)P_{\theta}+\epsilon Q:\theta\in\Theta_j,Q)\}\subset\{P:\TV(P,P_{\theta_t})\leq \epsilon+\delta+L\delta/4\}$$
for any $\theta_t$ such that $\TV(P_{\theta_t},P_{\theta_j})\leq L\delta/4$. Combining the bounds above, the proof is complete.
\end{proof}

\subsection{Proofs in Section \ref{sec:application}}\label{sec:6.2}

First, we give a lemma that establishes the equivalence between total variation distance and $\ell_2$ norm for linear regression and trace regression.
\begin{lemma}\label{lem:TV}
Assume the upper sparse eigenvalue condition in (\ref{eq:upperRE}) holds.
For $P_{\theta}$ specified in Section \ref{sec:reg}, there are constants $C_1,C_2$, such that
$$C_1\frac{\norm{\Sigma^{1/2}(\theta-\theta')}}{\sigma}\leq\TV(P_{\theta},P_{\theta'})\leq C_2\frac{\norm{\Sigma^{1/2}(\theta-\theta')}}{\sigma},$$
for any $\theta,\theta'\in\Theta(s,M,\sigma,\kappa)$. Similarly, assume the restricted isometry condition in (\ref{eq:RI}) holds. For $P_A$ specified in Section \ref{sec:rank}, there are constants $C_1,C_2$, such that
$$C_1\frac{\norm{\Sigma^{1/2}(\text{vec}(A)-\text{vec}(A'))}}{\sigma}\leq \TV(P_{A},P_{A'})\leq C_2\frac{\norm{\Sigma^{1/2}(\text{vec}(A)-\text{vec}(A'))}}{\sigma},$$
for any $A,A'\in\mathcal{A}(r,M,\sigma,\kappa)$.
\end{lemma}
\begin{proof}
Since the proofs of the two inequalities are nearly identical, we only give details for the first one.
The density function of $P_{\theta}$ is
$$(2\pi)^{-p/2}|\Sigma|^{-1/2}e^{-\frac{1}{2}X^T\Sigma^{-1}X}\times \frac{1}{\sqrt{2\pi\sigma^2}}e^{-\frac{1}{2\sigma^2}(y-X^T\theta)^2},$$
where $|\Sigma|$ is the determinant of $\Sigma$.
Therefore, by the definition of total variation distance, we have
\begin{eqnarray*}
&&\TV(P_{\theta},P_{\theta'})\\
&=&P_{\theta}\left\{(y-X^T\theta)^2<(y-X^T\theta')^2\right\}-P_{\theta'}\left\{(y-X^T\theta)^2<(y-X^T\theta')^2\right\}.
\end{eqnarray*}
Note that
\begin{eqnarray*}
&& P_{\theta}\left\{(y-X^T\theta)^2<(y-X^T\theta')^2\right\} \\
&=& P_{\theta}\left\{\frac{(y-X^T\theta)}{\sigma}\frac{(X^T(\theta-\theta'))}{|(X^T(\theta-\theta'))|}>-\frac{|X^T(\theta-\theta')|}{2\sigma}\right\} \\
&=& \mathbb{E}\Phi\left(\frac{|X^T(\theta-\theta')|}{2\sigma}\right),
\end{eqnarray*}
where $\Phi$ is the cumulative distribution function of $N(0,1)$ and the last equality is because $\frac{(y-X^T\theta)}{\sigma}\frac{(X^T(\theta-\theta'))}{|(X^T(\theta-\theta'))|}$ is distributed by $N(0,1)$ conditioning on $X$. Hence,
\begin{equation}
\TV(P_{\theta},P_{\theta'})=2\mathbb{E}\Phi\left(\frac{|X^T(\theta-\theta')|}{2\sigma}\right)-1=\mathbb{E}\int_{-\frac{|X^T(\theta-\theta')|}{2\sigma}}^{\frac{|X^T(\theta-\theta')|}{2\sigma}}\frac{1}{\sqrt{2\pi}}e^{-\frac{t^2}{2}}dt.\label{eq:TVV}
\end{equation}
An upper bound for (\ref{eq:TVV}) is
$$\frac{1}{\sqrt{2\pi}}\mathbb{E}\frac{|X^T(\theta-\theta')|}{\sigma}=\frac{\norm{\Sigma^{1/2}(\theta-\theta')}}{\sigma\sqrt{2\pi}}\mathbb{E}|Z|,$$
for $Z\sim N(0,1)$.
A lower bound for (\ref{eq:TVV}) is
\begin{eqnarray*}
&& \mathbb{E}\frac{1}{\sqrt{2\pi}}e^{-\frac{|{X^T(\theta-\theta')}|^2}{8\sigma^2}}\frac{|{X^T(\theta-\theta')}|}{\sigma} \\
&=& \frac{\norm{\Sigma^{1/2}(\theta-\theta')}}{\sigma\sqrt{2\pi}}\mathbb{E}e^{-\frac{\norm{\Sigma^{1/2}(\theta-\theta')}^2}{8\sigma^2}|Z|^2}|Z| \\
&\geq& \frac{\norm{\Sigma^{1/2}(\theta-\theta')}}{\sigma\sqrt{2\pi}}\mathbb{E}e^{-C^2M^2|Z|^2/2}|Z|,
\end{eqnarray*}
where $Z\sim N(0,1)$ and the last inequality follows from the upper sparse eigenvalue condition in (\ref{eq:upperRE}) and the fact $\|\theta\|\leq M\sigma/\kappa$. Hence, we have proved that
$$C_1\frac{\norm{\Sigma^{1/2}(\theta-\theta')}}{\sigma}\leq\TV(P_{\theta},P_{\theta'})\leq C_2\frac{\norm{\Sigma^{1/2}(\theta-\theta')}}{\sigma},$$
with $C_1=\frac{1}{\sqrt{2\pi}}\mathbb{E}e^{-C^2M^2|Z|^2/2}|Z|$ and $C_2=\frac{1}{\sqrt{2\pi}}\mathbb{E}|Z|$.
\end{proof}
With the help of the above lemma, we are ready to give the proofs of the results in Section \ref{sec:application}.

\begin{proof}[Proof of Corollary \ref{cor:density}]
The result directly follows Theorem \ref{thm:estimator} by realizing that $\TV(P_{f_1},P_{f_2})=\frac{1}{2}\norm{f_1-f_2}_1$.
\end{proof}

\begin{proof}[Proof of Corollary \ref{cor:reg}]
We use the estimator (\ref{eq:estimator}) with $\delta=\sqrt{\frac{s\log\frac{ep}{s}}{n}}$. Here we work with a $\delta$-packing set $\Theta'=\{\theta_1,...,\theta_m\}$ of maximum cardinality in the sense that $\min_{i\neq j}\TV(P_{\theta_i},P_{\theta_j})\geq \delta$ with the largest possible $m$. The value $m=\mathcal{N}(\delta,\Theta,\TV(\cdot,\cdot))$ is called $\delta$-packing number. It is easy to see that $\Theta'$ is also a $\delta$-covering set and $m$ is equal to $\delta$-covering number up to a constant factor. Indeed, $\delta$-covering and $\delta$-packing numbers are (up to a constant factor) essentially the same, i.e., $\mathcal{M}(\delta,\Theta,\TV(\cdot,\cdot))\leq\mathcal{N}(\delta,\Theta,\TV(\cdot,\cdot))\leq \mathcal{M}(\delta/2,\Theta,\TV(\cdot,\cdot))$. See, for example, \cite{pollard1990empirical}. According to Lemma \ref{lem:TV}, we have that $\min_{i\neq j}\norm{\Sigma^{1/2}(\theta_i-\theta_j)}\geq \sigma\delta/(C_2)$. Hence, for any $\theta_0$,
we have
\begin{eqnarray*}
&& \left|\{\theta\in\Theta': l\delta<\TV(P_{\theta},P_{\theta_0})\leq (l+1)\delta\}\right| \\
&\leq& \left|\{\theta\in\Theta': \TV(P_{\theta},P_{\theta_0})\leq (l+1)\delta\}\right| \\
&\leq& \sum_{|S|\leq s}\left|\{\theta\in\Theta': \supp(\theta)=S, \TV(P_{\theta},P_{\theta_0})\leq (l+1)\delta\}\right| \\
&\leq& \sum_{|S|\leq s}\left|\{\theta\in\Theta': \supp(\theta)=S, \norm{\Sigma^{1/2}(\theta-\theta_0)}\leq \sigma(l+1)\delta/C_1\}\right| \\
&\leq& \exp\left(s\log\frac{ep}{s}\right) (l+1)^{C_3s},
\end{eqnarray*}
where the last inequality is through a volume ratio argument \citep{pollard1990empirical}. Taking supreme over $\theta_0$, we have
$$\log D_l(\delta)\leq C_4\left(s\log\frac{ep}{s}+s\log(l+1)\right).$$
Using Theorem \ref{thm:local} with $L=\floor{C_5\frac{\delta+\epsilon}{\delta}}$ for some large $C_5>0$, direct calculation gives that
$$\TV(P_{\hat{\theta}},P_{\theta})\leq C_6(\delta+\epsilon),$$
for some $C_6>0$, with probability at least $1-\exp\left(-C'n(\delta^2+\epsilon^2)\right)$, where $\delta=\sqrt{\frac{s\log\frac{ep}{s}}{n}}$. By Lemma \ref{lem:TV} and the definition of $\kappa$, we obtain the convergence rate with the desired loss functions. Thus, the proof is complete.
\end{proof}

\begin{proof}[Proof of Corollary \ref{cor:rank}]
We use the estimator (\ref{eq:estimator}) with $\delta=\sqrt{\frac{r(p_1+p_2)}{n}}$. Similar to the argument in the proof of Corollary \ref{cor:reg}, we work with a $\delta$-packing set $\mathcal{A}'=\{A_1,...,A_m\}$ of maximum cardinality in the sense that $\min_{i\neq j}\TV(P_{A_i},P_{A_j})\geq \delta$ with the largest possible $m$. It is easy to see $\mathcal{A}'$ is also a $\delta$-covering set. According to Lemma \ref{lem:TV}, we have that $\min_{i\neq j}\norm{\Sigma^{1/2}(\text{vec}(A_i)-\text{vec}(A_j))}\geq \sigma\delta/(C_2)$ and consequently $\min_{i\neq j}\fnorm{A_i-A_j}\geq C_3\sigma\delta/\kappa$ from the restricted isometry condition in (\ref{eq:RI}). Hence,
we have
\begin{eqnarray*}
&& \left|\{A\in\mathcal{A}': l\delta<\TV(P_{A},P_{A_0})\leq (l+1)\delta\}\right| \\
&\leq& \left|\{A\in\mathcal{A}': \TV(P_{A},P_{A_0})\leq (l+1)\delta\}\right| \\
&\leq& \left|\{A\in\mathcal{A}':\norm{\Sigma^{1/2}(\text{vec}(A)-\text{vec}(A_0))}\leq\sigma(l+1)\delta/C_1\}\right| \\
&\leq& \left|\{A\in\mathcal{A}':\fnorm{A-A_0}\leq C_4\sigma(l+1)\delta/\kappa\}\right| \\
&\leq& (l+1)^{C_5r(p_1+p_2)},
\end{eqnarray*}
where the second inequality follows from Lemma \ref{lem:TV}, the third inequality follows from the restricted isometry condition in (\ref{eq:RI}), and the last inequality is due to Lemma 3.1 of \cite{candes2011tight} and the fact $\min_{i\neq j}\fnorm{A_i-A_j}\geq C_3\sigma\delta/\kappa$. Taking supreme over $\theta_0$, we have
$$\log D_l(\delta)\leq C_5r(p_1+p_2)\log(l+1).$$
Using Theorem \ref{thm:local} with $L=\floor{C_6\frac{\delta+\epsilon}{\delta}}$ for some large $C_6>0$, direct calculation gives that
$$\TV(P_{\hat{A}},P_{A})\leq C_7(\delta+\epsilon),$$
for some $C_7>0$, with probability at least $1-\exp\left(-C'n(\delta^2+\epsilon^2)\right)$, where $\delta=\sqrt{\frac{r(p_1+p_2)}{n}}$. By Lemma \ref{lem:TV} and the definition of $\kappa$, we obtain the convergence rate with the desired loss function. Thus, the proof is complete.
\end{proof}

\subsection{Proofs in Section \ref{sec:sup}}\label{sec:6.3}

Before stating the proofs of Theorem \ref{thm:sup} and Corollary \ref{cor:sup}, we present a lemma that establishes equivalence between different loss functions.

\begin{lemma}\label{lem:compare}
For $P_f$ specified in Section \ref{sec:sup}, there are constants $C_1,C_2,C_3,C_4$, such that
$$C_2^{-1}\norm{f_1-f_2}_{\infty}\leq \sum_{l\geq 0}2^{l/2}\max_{0\leq k\leq 2^l-1}|f_{1,lk}-f_{2,lk}|\leq C_1^{-1}\norm{f_1-f_2}_{\infty},$$
$$C _3\norm{f_1-f_2}\leq \TV(P_{f_1},P_{f_2})\leq C _4\norm{f_1-f_2},$$
for all $f_1,f_2\in\mathcal{H}(\beta,L)$, where $\{f_{1,lk}\}$ and $f_{2,lk}$ are wavelet coefficients of $f_1$ and $f_2$, and $\norm{\cdot}$ is understood as both vector and function $\ell_2$ norm.
\end{lemma}
\begin{proof}
It is well known that two term $\sum_{l\geq 0}2^{l/2}\max_{0\leq k\leq 2^l-1}|f_{1,lk}-f_{2,lk}|$ and $\norm{f_1-f_2}_{\infty}$ are equivalent in the wavelet literature. See, for example, \cite{hoffmann2015adaptive}. The equivalence implies that $\mathcal{H}(\beta,L)$ is a subset of an $\ell_2$ ball. Indeed, for any $f\in\mathcal{H}(\beta,L)$, we have that $\norm{f}\leq \norm{f}_{\infty}$, which implies
\begin{equation}
\norm{f}\leq C_2\sum_{l\geq 0}2^{l/2}\max_{0\leq k\leq 2^l-1}|f_{lk}|\leq C_2\sum_{l\geq 0}2^{l/2}M2^{-l(1/2+\beta)}\leq \frac{1}{1-2^{-\beta}}C_2M.\label{eq:ball}
\end{equation}
To study $\TV(P_{f_1},P_{f_2})$, we use an equivalent model of (\ref{eq:white_1}) in terms of wavelet coefficients. That is,
\begin{equation}
y_{lk}=f_{lk}+z_{lk},\quad l\geq 0, 0\leq k\leq 2^l-1,\label{eq:wave-coef}
\end{equation}
where $\{z_{lk}\}$ are i.i.d. $N(0,1)$. Then, direct calculation gives
\begin{equation}
\TV(P_{f_1},P_{f_2}) = 2\Phi\left(\frac{\norm{f_1-f_2}}{2}\right)-1=\int_{-\frac{\norm{f_1-f_2}}{2}}^{\frac{\norm{f_1-f_2}}{2}}\frac{1}{\sqrt{2\pi}}e^{-\frac{t^2}{2}}dt.\label{eq:TV-sup}
\end{equation}
An upper bound for (\ref{eq:TV-sup}) is $\frac{1}{\sqrt{2\pi}}\norm{f_1-f_2}$. A lower bound for (\ref{eq:TV-sup}) is
$$\frac{1}{\sqrt{2\pi}}e^{-\frac{\norm{f_1-f_2}^2}{8}}\norm{f_1-f_2}\geq \frac{1}{\sqrt{2\pi}}e^{-\frac{C_2^2M^2}{2(1-2^{-\beta})^2}}\norm{f_1-f_2},$$
where we have used (\ref{eq:ball}). Thus, the proof is complete.
\end{proof}

The next lemma characterizes the statistical property of a median estimator under Huber's $\epsilon$-contamination model.

\begin{lemma}\label{lem:median}
Assume $\epsilon<1/4$.
There exists a constant $C>0$, such that for each $0\leq l\leq L$ and $0\leq k\leq 2^l-1$, we have
$$\sup_{f\in\mathcal{H}(\beta,M),Q}\mathbb{P}_{(\epsilon,f,Q)}\left\{|\hat{f}_{lk}-f_{lk}|>C\left(\sqrt{\frac{\log(1/\delta)}{n}}\vee\epsilon\right)\right\}\leq 2\delta,$$
for any $\delta>0$ that $\sqrt{\frac{\log(1/\delta)}{n}}$ is sufficiently small.
\end{lemma}
\begin{proof}
Since ${y}_{lk,i}\sim N(f_{lk},1)$, the setting is a special case of Theorem 2.1 in \cite{chen2015robust}. A careful examination of its proof gives the desired result.
\end{proof}

Now we give the proofs of \ref{thm:sup} and Corollary \ref{cor:sup} with the facility of the above two lemmas.

\begin{proof}[Proof of Theorem \ref{thm:sup}]
Note that
$$\sum_{l\geq 0}2^{l/2}\max_{0\leq k\leq 2^l-1}|\hat{f}_{lk}-f_{lk}|=\sum_{l\leq L}2^{l/2}\max_{0\leq k\leq 2^l-1}|\hat{f}_{lk}-f_{lk}|+\sum_{l>L}2^{l/2}\max_{0\leq k\leq 2^l-1}|f_{lk}|.$$
It is sufficient to give upper bounds for the two terms. Since $f\in\mathcal{H}(\beta,M)$,
$$\sum_{l>L}2^{l/2}\max_{0\leq k\leq 2^l-1}|f_{lk}|\leq \sum_{l>L}2^{l/2}M2^{-l(1/2+\beta)}\leq \frac{2M}{1-2^{-\beta}}\left(\frac{\log n}{n}\vee\epsilon^2\right)^{\frac{\beta}{2\beta+1}},$$
by the definition of $L$. Using Lemma \ref{lem:median} with $2\delta=\exp\left(-C'(\log n+n\epsilon^2)\right)$ for some constant $C'>0$ and a union bound argument, we have
$$\max_{l\geq L,0\leq k\leq 2^l-1}|\hat{f}_{lk}-f_{lk}|\leq \bar{C}\left(\sqrt{\frac{\log n}{n}}\vee\epsilon\right),$$
with probability at least $1-\exp\left(-C'(\log n+n\epsilon^2)\right)$. Therefore,
$$\sum_{l\leq L}2^{l/2}\max_{0\leq k\leq 2^l-1}|\hat{f}_{lk}-f_{lk}|\leq\bar{C}\left(\sqrt{\frac{\log n}{n}}\vee\epsilon\right)\sum_{l\leq L}2^{l/2}\leq \tilde{C}\left(\frac{\log n}{n}\vee\epsilon^2\right)^{\frac{\beta}{2\beta+1}}.$$
Hence,
$$\sum_{l\geq 0}2^{l/2}\max_{0\leq k\leq 2^l-1}|\hat{f}_{lk}-f_{lk}|\leq C\left(\frac{\log n}{n}\vee\epsilon^2\right)^{\frac{\beta}{2\beta+1}},$$
with probability at least $1-\exp\left(-C'(\log n+n\epsilon^2)\right)$. By Lemma \ref{lem:compare}, the same bound holds for $\norm{\hat{f}-f}_{\infty}$, and the proof is complete.
\end{proof}

\begin{proof}[Proof of Corollary \ref{cor:sup}]
The lower bound $\mathcal{R}(0)\vee\omega(\epsilon,\mathcal{H}(\beta,M))$ immediately follows from Theorem \ref{thm:lower}. In this problem, it is known that $\mathcal{R}(0)\asymp \left(\frac{n}{\log n}\right)^{-\frac{2\beta}{2\beta+1}}$. See, for example, \cite{tsybakov2008introduction}. Therefore, it is sufficient to calculate the modulus of continuity $\omega(\epsilon,\mathcal{H}(\beta,L))$. Define $\bar{l}$ to be the greatest integer such that $2^{\bar{l}(1/2+\beta)}\epsilon\leq M$. Then, let $f_1=0$ and $f_2=f_1+\epsilon\psi_{\bar{l}1}$. It is easy to see that $f_1,f_2\in\mathcal{H}(\beta,M)$. By Lemma \ref{lem:compare}, $\TV(P_{f_1},P_{f_2})\leq C_4\norm{f_1-f_2}=(2\pi)^{-1/2}\epsilon\leq \epsilon/(1-\epsilon)$, where $C_4=(2\pi)^{-1/2}$ according to the proof of Lemma \ref{lem:compare}. Moreover, we have that $$\norm{f_1-f_2}_{\infty}\geq C_1\sum_{l\geq 0}2^{l/2}\max_{0\leq k\leq 2^l-1}|f_{1,lk}-f_{2,lk}|\geq C_12^{\bar{l}/2}\epsilon\gtrsim \epsilon^{\frac{2\beta}{2\beta+1}}.$$ Hence, $\omega(\epsilon,\mathcal{H}(\beta,M))\gtrsim \epsilon^{\frac{2\beta}{2\beta+1}}$, and the proof is complete.
\end{proof}

%% file: manuscript.bbl
\newcommand{\SortNoop}[1]{}
\begin{thebibliography}{25}
\providecommand{\natexlab}[1]{#1}
\providecommand{\url}[1]{\texttt{#1}}
\expandafter\ifx\csname urlstyle\endcsname\relax
  \providecommand{\doi}[1]{doi: #1}\else
  \providecommand{\doi}{doi: \begingroup \urlstyle{rm}\Url}\fi

\bibitem[Birg\'{e}(1984)]{birge1979th}
Lucien Birg\'{e}.
\newblock Sur un the\'{o}r\`{e}me de minimax et son application aux tests.
\newblock \emph{Probability and Mathematical Statistics}, 3:\penalty0 259--282,
  1984.

\bibitem[Brown and Low(1996)]{brown1996asymptotic}
Lawrence~D Brown and Mark~G Low.
\newblock Asymptotic equivalence of nonparametric regression and white noise.
\newblock \emph{The Annals of Statistics}, 24\penalty0 (6):\penalty0
  2384--2398, 1996.

\bibitem[Candes and Plan(2011)]{candes2011tight}
Emmanuel~J Candes and Yaniv Plan.
\newblock Tight oracle inequalities for low-rank matrix recovery from a minimal
  number of noisy random measurements.
\newblock \emph{Information Theory, IEEE Transactions on}, 57\penalty0
  (4):\penalty0 2342--2359, 2011.

\bibitem[Chen et~al.(2015)Chen, Gao, and Ren]{chen2015robust}
Mengjie Chen, Chao Gao, and Zhao Ren.
\newblock Robust covariance matrix estimation via matrix depth.
\newblock \emph{arXiv preprint arXiv:1506.00691}, 2015.

\bibitem[Cohen et~al.(1993)Cohen, Daubechies, and Vial]{cohen1993wavelets}
Albert Cohen, Ingrid Daubechies, and Pierre Vial.
\newblock Wavelets on the interval and fast wavelet transforms.
\newblock \emph{Applied and Computational Harmonic Analysis}, 1\penalty0
  (1):\penalty0 54--81, 1993.

\bibitem[Devroye and Lugosi(2012)]{devroye2012combinatorial}
Luc Devroye and G{\'a}bor Lugosi.
\newblock \emph{Combinatorial methods in density estimation}.
\newblock Springer Science \& Business Media, 2012.

\bibitem[Donoho(1994)]{donoho1994statistical}
David~L Donoho.
\newblock Statistical estimation and optimal recovery.
\newblock \emph{The Annals of Statistics}, 22\penalty0 (1):\penalty0 238--270,
  1994.

\bibitem[Donoho and Liu(1991)]{donoho1991geometrizing}
David~L Donoho and Richard~C Liu.
\newblock Geometrizing rates of convergence, {III}.
\newblock \emph{The Annals of Statistics}, 19\penalty0 (2):\penalty0 668--701,
  1991.

\bibitem[Hoffmann et~al.(2015)Hoffmann, Rousseau, and
  Schmidt-Hieber]{hoffmann2015adaptive}
Marc Hoffmann, Judith Rousseau, and Johannes Schmidt-Hieber.
\newblock On adaptive posterior concentration rates.
\newblock \emph{The Annals of Statistics}, 43\penalty0 (5):\penalty0
  2259--2295, 2015.

\bibitem[Huber(1964)]{huber1964robust}
Peter~J Huber.
\newblock Robust estimation of a location parameter.
\newblock \emph{The Annals of Mathematical Statistics}, 35\penalty0
  (1):\penalty0 73--101, 1964.

\bibitem[Huber(1965)]{huber1965robust}
Peter~J Huber.
\newblock A robust version of the probability ratio test.
\newblock \emph{The Annals of Mathematical Statistics}, 36\penalty0
  (6):\penalty0 1753--1758, 1965.

\bibitem[Huber and Strassen(1973)]{huber1973minimax}
Peter~J Huber and Volker Strassen.
\newblock Minimax tests and the {N}eyman-{P}earson lemma for capacities.
\newblock \emph{The Annals of Statistics}, 1\penalty0 (2):\penalty0 251--263,
  1973.

\bibitem[Koltchinskii et~al.(2011)Koltchinskii, Lounici, and
  Tsybakov]{koltchinskii2011nuclear}
Vladimir Koltchinskii, Karim Lounici, and Alexandre~B Tsybakov.
\newblock Nuclear-norm penalization and optimal rates for noisy low-rank matrix
  completion.
\newblock \emph{The Annals of Statistics}, 39\penalty0 (5):\penalty0
  2302--2329, 2011.

\bibitem[Le~Cam(1973)]{lecam1973convergence}
L~Le~Cam.
\newblock Convergence of estimates under dimensionality restrictions.
\newblock \emph{The Annals of Statistics}, 1\penalty0 (1):\penalty0 38--53,
  1973.

\bibitem[Neyman and Pearson(1933)]{neyman1933problem}
J~Neyman and ES~Pearson.
\newblock On the problem of the most efficient tests of statistical hypotheses.
\newblock \emph{Philosophical Transactions of the Royal Society of London A:
  Mathematical, Physical and Engineering Sciences}, 231\penalty0
  (694-706):\penalty0 289--337, 1933.

\bibitem[Nussbaum(1996)]{nussbaum1996asymptotic}
Michael Nussbaum.
\newblock Asymptotic equivalence of density estimation and {G}aussian white
  noise.
\newblock \emph{The Annals of Statistics}, 24\penalty0 (6):\penalty0
  2399--2430, 1996.

\bibitem[Pinsker(1980)]{pinsker1980optimal}
Mark~Semenovich Pinsker.
\newblock Optimal filtering of square-integrable signals in {G}aussian noise.
\newblock \emph{Problemy Peredachi Informatsii}, 16\penalty0 (2):\penalty0
  52--68, 1980.

\bibitem[Pollard(1990)]{pollard1990empirical}
David Pollard.
\newblock Empirical {P}rocesses: {T}heory and {A}pplications.
\newblock In \emph{NSF-CBMS regional conference series in probability and
  statistics}, pages i--86. JSTOR, 1990.

\bibitem[Raskutti et~al.(2011)Raskutti, Wainwright, and
  Yu]{raskutti2011minimax}
Garvesh Raskutti, Martin~J Wainwright, and Bin Yu.
\newblock Minimax rates of estimation for high-dimensional linear regression
  over $\ell_q$-balls.
\newblock \emph{IEEE Transactions on Information Theory}, 57\penalty0
  (10):\penalty0 6976--6994, 2011.

\bibitem[Recht et~al.(2010)Recht, Fazel, and Parrilo]{recht2010guaranteed}
Benjamin Recht, Maryam Fazel, and Pablo~A Parrilo.
\newblock Guaranteed minimum-rank solutions of linear matrix equations via
  nuclear norm minimization.
\newblock \emph{SIAM review}, 52\penalty0 (3):\penalty0 471--501, 2010.

\bibitem[Rohde and Tsybakov(2011)]{rohde2011estimation}
Angelika Rohde and Alexandre~B Tsybakov.
\newblock Estimation of high-dimensional low-rank matrices.
\newblock \emph{The Annals of Statistics}, 39\penalty0 (2):\penalty0 887--930,
  2011.

\bibitem[Rudelson and Zhou(2013)]{rudelson2013reconstruction}
Mark Rudelson and Shuheng Zhou.
\newblock Reconstruction from anisotropic random measurements.
\newblock \emph{IEEE Transactions on Information Theory}, 59\penalty0
  (6):\penalty0 3434--3447, 2013.

\bibitem[Tsybakov(2008)]{tsybakov2008introduction}
Alexandre~B Tsybakov.
\newblock \emph{Introduction to Nonparametric Estimation}.
\newblock Springer Science \& Business Media, 2008.

\bibitem[Yang and Barron(1999)]{yang1999information}
Yuhong Yang and Andrew Barron.
\newblock Information-theoretic determination of minimax rates of convergence.
\newblock \emph{The Annals of Statistics}, 27\penalty0 (5):\penalty0
  1564--1599, 1999.

\bibitem[Yatracos(1985)]{yatracos1985rates}
Yannis~G Yatracos.
\newblock Rates of convergence of minimum distance estimators and
  {K}olmogorov's entropy.
\newblock \emph{The Annals of Statistics}, 13\penalty0 (2):\penalty0 768--774,
  1985.

\end{thebibliography}
